\documentclass[10pt,a4paper]{article}
\usepackage[utf8]{inputenc}
\usepackage[T1]{fontenc}
\usepackage{amsmath}
\usepackage{amsfonts}
\usepackage{amssymb}
\usepackage{amsthm}
\usepackage{graphicx}
\usepackage{authblk}

\newcommand{\ep}{\varepsilon}
\newcommand{\bea}{\begin{eqnarray*}}
\newcommand{\eea}{\end{eqnarray*}}
\newtheorem{thm}{Theorem}[section]

\newtheorem{lem}{Lemma}[section]

\newtheorem{defn}{Definition}[section]

\title{The parameter-uniform convergence of a fitted operator method on non-uniform meshes for a singularly perturbed initial value problem}
\author{John J H Miller}
\affil{INCA-Institute for Numerical Computation and Analysis, Dublin, Ireland.\\email:  jjhmiller@gmail.com}

\date{}
\begin{document}
	\maketitle
	\begin{abstract}
The parameter-uniform convergence of a fitted operator method for a singularly perturbed differential equation is normally available only for uniform meshes. Here we establish the parameter-uniform convergence of a fitted operator method on a non-uniform mesh for a singularly perturbed initial value problem. This is obtained by a new method of proof.
	\end{abstract}
\section{Introduction} The original source of parameter-uniform convergence proofs for fitted operator finite difference methods is the paper of Il'in \cite{AMI}, which was used extensively in the monograph Doolan et al. \cite{DMS} and later by other authors. In these works the proofs of parameter-uniform convergence are restricted to a uniform finite difference mesh. The point of the present paper is to remove this restriction on the mesh by using an entirely new method of proof.
\section{The continuous problem} 
We consider the simple singularly perturbed first order initial value problem $(P_{\ep})$  on the  interval  $\Omega=(0,T]$  
\bea
\ep u_{\ep}^{\prime}  + a(t)u_{\ep}&=&f (t), \;\;t\in \Omega,\\
u_{\ep}(0) \;\;\mathrm{given},
\eea
where, for all $t \in \Omega,\;\;  a(t)\geq \alpha >0$ and $0 < \ep \leq 1$. \\
We are interested in designing  a numerical method which gives good approximations to the solution of $(P_{\ep})$, regardless of the value of $\ep$ in the entire range $0< \ep \leq 1$. It can be seen from the exact solution of this problem that as $\ep \rightarrow 0$ the gradient of the solution becomes increasingly steep in a neighbourhood of $t=0$. This phenomenon is called a {\it layer}. In this case it is an initial layer, because it is associated with a  boundary point where there is an initial condition. In the analysis that follows it becomes clear that the width of this layer is $O(\ep)$.  It is important to note that henceforth $C$ denotes a generic constant which is independent of the singular perturbation parameter $\ep$.\\
To analyse such problems and their numerical solutions we introduce some norms and semi-norms. For singular perturbation problems it is important to work in the maximum norm.
\noindent We define the maximum norm of a differentiable function on a set $S$ by
\begin{defn}
	$|\phi|_{S}=\sup_{t\in S}|\phi(t)|$
\end{defn}
\noindent and for any positive integer $k,$ we define the $k$-th order semi-norm of a differentiable function on a set $S$ by
\begin{defn}
	$|\phi|_{k,S}=\sup_{t\in S}|\phi^{(k)}(t)|.$
\end{defn}
\noindent For $k=0$, the semi-norm becomes the maximum norm.  Note that when the meaning is clear the subscript $S$ is usually dropped from the notation.\\
\noindent For some $T>0$, let $\Omega=(0,T]$ and $\overline{\Omega}=[0,T].$\\
For convenience we introduce the differential operator
\bea
L_{\ep}=\ep\frac{d}{dt}+a(t).
\eea
The operator ${L_{\ep}}$ satisfies the following maximum principle
\begin{lem} \label{max}
	Let $ \psi(t)$ be any function in the domain of ${L_{\ep}}$ such
	that $\psi(0)\ge 0.$ Then ${L_{\ep}}{\psi}(t) \geq 0$ for
	all $t \in \Omega$ implies that ${\psi}(t) \geq 0$ for all $t
	\in \overline{\Omega}$.
\end{lem}
\begin{proof} Let $t^*$ be such that $\psi(t^{*})=\min_{t}\psi(t)$
	and assume that the lemma is false. Then $\psi(t^{*})<0.$ 
	From the hypotheses we have $t^*\neq 0$ and
	$\psi^{\prime}(t^*)\leq 0$. Thus
	\bea
	L_{\ep}\psi(t^*)=\ep\psi^{\prime}(t^*)+a(t^*)\psi(t^*)<0,
	\eea
	which contradicts the assumption.
\end{proof}
\noindent This leads immediately to the following stability result.
\begin{lem}\label{stab}
	Let $ \psi(t)$ be any function in the domain of ${L_{\ep}}.$
	Then
	\bea
	|\psi(t)|\le \max\{|
	\psi(0)|,\frac{1}{\alpha}|L_{\ep}\psi|\}  \;\; t\in \overline{\Omega}
	\eea
\end{lem}
\begin{proof}Define the two functions
	\[{\theta}^{\pm}(t)=\max\{|\psi(0)|,\frac{1}{\alpha}|L_{\ep}\psi|\}\pm
	{\psi}(t). \] It is not hard to
	verify that ${\theta}^{\pm}(0)\geq 0$ and
	${L_{\ep}}{\theta}^{\pm}(t)\geq 0$. It follows from Lemma
	\ref{max} that ${\theta}^{\pm}(t)\geq 0,$ for all $t \in \overline{\Omega},$ as required.
\end{proof}
In order to sharpen the classical bounds on the solution
we introduce the Shishkin decomposition of the problem.\\
First we define the reduced problem $(P_0)$ corresponding to $(P_{\ep})$ as 
\bea
a(t)u_0(t)=f(t).
\eea
This is obtained by putting $\ep=0$ in $(P_{\ep})$. Its solution is clearly
\[u_0(t)=\frac{f(t)}{a((t)}.\]
We then define the Shishkin decomposition of $u_{\ep}$. 
We write 
\bea
u_{\ep}=v_{\ep}+w_{\ep},
\eea
where $v_{\ep}$ is the smooth component of the decomposition and is defined to be the solution of the problem
\bea
L_{\ep} v_{\ep}=f,  \;\; v_{\ep}(0)=u_0 (0);
\eea
$u_0$ being the solution of the reduced problem. It follows that the singular component $w_{\ep}$ must be the solution of the problem
\bea
L_{\ep} w_{\ep}=0,  \;\; w_{\ep}(0)=u_{\ep} (0) - v_{\ep} (0).
\eea
The following observations are used in the proof of the next lemma.
Note that the equation for $v_{\ep}$ gives $\ep v_{\ep}^{\prime}(0)+a(0)v_{\ep}(0)=f(0)$ and the initial condition for  $v_{\ep}$ gives $a(0)v_{\ep}(0)=f(0)$. Combining these we obtain $\ep v_{\ep}^{\prime}(0)=0$ and so $v_{\ep}^{\prime}(0)=0$.\\ Note also that $L_{\ep}v_{\ep}^{\prime}=\ep v_{\ep}^{\prime\prime}+av_{\ep}^{\prime}$ and, by differentiating the equation satisfied by $v_{\ep}$, we get $\ep v_{\ep}^{\prime\prime}+(av_{\ep})^{\prime}=f^{\prime}$. Eliminating $v_{\ep}^{\prime\prime}$ gives $L_{\ep}v_{\ep}^{\prime}=f^{\prime}-a^{\prime}v_{\prime}$.\\
The sharper bounds are contained in the following lemma.
\begin{lem}\label{sharp}
	The components $v_{\ep},\; w_{\ep}$ of the exact solution $u_{\ep}$ satisfy the bounds
	\bea
	|v_{\ep}|_k &\leq& C, \;\; k=0, 1 \;\; \mathrm{and}\;\; |v_{\ep}|_2\leq C\ep^{-1},\\
	|w_{\ep}^{(k)}(t)| &\leq & C \ep^{-k} e^{-\frac{\alpha t}{\ep}}, \;\; k=0, 1, 2 \mathrm{ \;for\; all\;\;}  t \in \overline{\Omega}. 
	\eea
\end{lem}
\noindent{\bf Proof} 
The bounds on $v_{\ep}$ and its derivatives follow from the above observations and Lemma \ref{stab}.\\ The bound on $v_{\ep}^{\prime}$ is obtained by considering the two functions $\phi^{\pm}(t)=C(1+|f^{\prime}|) \pm v_{\ep}^{\prime}(t).$ Then, 
$\phi^{\pm}(0)=C(1+|f^{\prime}|) \pm v_{\ep}^{\prime}(0)\ge v_{\ep}^{\prime}(0)=0$ . Furthermore, $L_{\ep}\phi^{\pm}(t)=a(t)C(1+|f^{\prime}|)\pm L_{\ep}v_{\ep}^{\prime}(t) \ge C\alpha (1+|f^{\prime}|) \pm (f^{\prime}-a^{\prime}v_{\prime}) \ge 0$. The required bound on $v_{\ep}^{\prime}$ follows from the bound on $v_{\ep}$ and the maximum principle Lemma \ref{max}. The bound on $v_{\ep}^{\prime \prime}$ is obtained from the equation $\ep v_{\ep}^{\prime\prime}+(av_{\ep})^{\prime}=f^{\prime}$ and the bounds on $v_{\ep}$ and $v_{\ep}^{\prime}$.

The bound on $w_{\ep}$ is obtained similarly. We introduce the functions
\bea
\psi^{\pm}(t)=Ce^{-\frac{\alpha t}{\ep}} \pm w_{\ep}(t),
\eea
where $C$ is a suitably large constant. Then 
\bea
\psi^{\pm}(0)=C \pm w_{\ep}(0)=C \pm (u_{\ep}(0)-v_{\ep}(0))\geq C-(|u_0|+ |v_0(0)|)\geq 0.
\eea
Also, 
\bea
L_{\ep} \psi^{\pm}(t)=CL_{\ep}  e^{-\frac{\alpha t}{\ep}} \pm L_{\ep}w_{\ep}(t)=CL_{\ep}  e^{-\frac{\alpha t}{\ep}} =(a(t)-\alpha)e^{-\frac{\alpha t}{\ep}}>0.
\eea
From the maximum principle we then have 
\bea
\psi^{\pm}(t)\geq 0 \;\; \mathrm{for \; all} \;\; t \in \overline{\Omega},
\eea
and so 
\bea
|w_{\ep}^{(k)}(t)| \leq  C \ep^{-k} e^{-\frac{\alpha t}{\ep}}  \;\; \mathrm{for \; all} \;\; t \in \overline{\Omega},
\eea
as required.\\
To bound the derivatives of $w_{\ep}$ we use the differential equation repeatedly. 
We have
\bea
w_{\ep}^{\prime}(t)=-\ep^{-1}a(t)w_{\ep}(t)
\eea
and so
\bea
|w_{\ep}^{\prime}(t)|=|\ep^{-1}a(t)w_{\ep}(t)|\leq C\ep^{-1}  e^{-\frac{\alpha t}{\ep}}.
\eea
Similarly
\bea
w_{\ep}^{\prime \prime}(t)=-\ep^{-1}(a(t)w_{\ep}^{\prime}(t)+a^{\prime}(t)w_{\ep}(t))
\eea
and so
\bea
|w_{\ep}^{\prime \prime}(t)|\leq C\ep^{-2}  e^{-\frac{\alpha t}{\ep}},
\eea
as required.\\
This completes the proof of this lemma.\\
\section{The discrete problem}
In order to discuss numerical solutions we discretise the domain $\Omega=(0,T]$ using a non-uniform mesh having $N$ sub-intervals of lengths $h_j$, which is determined by a set of $N+1$ points $\Omega^N=\{t_j\}_{j=0}^N$.  Here, $t_0=0, \;\; t_N=T$, and, for any $j, \; 1 \leq j \leq N, \;\; h_j=t_j-t_{j-1}.$
\\We now introduce the finite difference operators $ D^+_t \; \; D^{-}_t \;\;\delta^2_t$, where, on an arbitrary mesh  $\Omega^N = \{t_i\}_{i=0} ^{N}$,
\[D^+{U}(t_j)\;=\;\frac{{U}(t_{j+1})-{U}(t_j)}{t_{j+1}-t_j},\]
\[D^-{U}(t_j)\;=\;\frac{{U}(t_j)-{U}(t_{j-1})}{t_j-t_{j-1}}.\]
\[\delta^2{U}(t_j)\;=\;\frac{D^+{U}(t_j)-D^-{U}(t_j)}{(t_{j+1}-t_{j-1})/2}.\]
Using these we can define our fitted  backward Euler finite difference scheme on $\Omega^N$
\[\ep\sigma_j(\rho_j)  D^{-}{U^N} +a_j{U^N}=f_j,  \qquad {U^N}(0)={u}(0), \]
or  in operator form
\[{L}_{\sigma}^N {U^N} ={f},  \qquad  {U^N}(0)={u}(0), \] where
\[L^N_{\sigma}=\ep \sigma_j(\rho_j)  D^{-}+a_jI ,\]
$D^-$ is the backward difference operator and $\sigma_j(\rho_j)=a_j\rho_j/(e^{a_j\rho_j}-1)$ with $a_j=a(t_j)$ and $\rho_j=h_j/\ep$.\\
It is not hard to verify that for all $0<\rho_j <\infty$,  all $t_j\in \Omega^N$ and all positive numbers $p,\;q$,
\bea
0 < &\sigma_j(\rho_j)& <1,\\
|\sigma_j(\rho_j)-1| & \leq & \min\{1,C\rho_j\},\\
|e^{-p}-e^{-q}| & \leq & |p-q|e^{-\min\{p,q\}}.
\eea

We have the following discrete maximum principle for $L^N_{\sigma},$ analogous to the
continuous case.
\begin{lem}\label{dmax} 
	For any mesh function $\Psi^N$, the inequalities $
	{\Psi^N}(0)\;\ge\;0 \;\rm{and}\\
	\;{L_{\sigma}}^N
	{\Psi^N}(t_j)\ge 0\;$ for $1\;\le\;j\;\le\;N,\;$ imply
	that $\;\Psi^N(t_j)\ge 0\;$ for $0\;\le\;j\;\le\;N.\;$
\end{lem}
\begin{proof} Let $ j^*$ be such that
	$\Psi^N(t_{j^{*}})=\min_{j}\Psi^N(t_j)$ and assume that the lemma
	is false. Then $\Psi^N(t_{j^{*}})<0$ . From the hypotheses we
	have $j^*\neq 0$ and $\Psi^N (t_{j^*})-\Psi^N(t_{j^*-1})\leq 0$.
	Thus
	\begin{eqnarray*}{L^N_{\sigma}}\Psi^N(t_{j^*})=
		\sigma_{j^*}\frac{\Psi^N (t_{j^*})-\Psi^N(t_{j^*-1})}{t_{j^*} -t_{j^*-1}}+
		a_{j^*}\Psi^N (t_{j^*})\leq a_{j^*}\Psi^N (t_{j^*})<0,
	\end{eqnarray*} which contradicts the assumption, as required. \end{proof}

An immediate consequence of this is the following discrete stability
result, analogous to the continuous result.
\begin{lem}\label{dstab}
	For any mesh function $\Psi^N $, we have
	\[|\Psi^N(t_j)|\;\le\;\max\{|\Psi^N (0)|,\frac{1}{\alpha}|L^N_{\ep} \Psi^N|\},\;\; 0\leq j \leq N. \]
\end{lem}
\begin{proof} Define the two mesh functions
	\[{\Theta}^N_{\pm}(t)=\max\{|\Psi^N (0)|,\frac{1}{\alpha}|L_{\sigma}^N\Psi^N| \}\pm
	\Psi^N(t).\]
	It is not hard to verify that
	${\Theta}^N_{\pm}(0)\geq 0$ and
	${L_{\ep}^N}{\Theta}^N_{\pm}(t_j)\geq 0$. It follows from Lemma
	\ref{dmax} that ${\Theta}^N_{\pm}(t_j)\geq 0$ for all $0\leq j
	\leq N$.
\end{proof}
The Shishkin decomposition of the discrete solution is analogous to that of the exact solution. We have
\bea
U_{\ep}^N=V_{\ep}^N+W_{\ep}^N,
\eea
where $V^N_{\ep}$ is the smooth component of the decomposition and is defined to be the solution of the problem
\bea
L_{\sigma}^N V_{\ep}^N=f,  \;\; V^N_{\ep}(0)=v_{\ep} (0).
\eea
It follows that the singular component $W^N_{\ep}$ must be the solution of the problem
\bea
L_{\sigma} W^N_{\ep}=0,  \;\; W^N_{\ep}(0)=w_{\ep} (0).
\eea
\section{Parameter uniform convergence}
We now define what is meant by a parameter-uniform numerical method for a family of singular perturbation problems.
\begin{defn}
	Consider a family of problems $(P_{\ep})$ parameterised by the singular perturbation parameter $\ep,\;\;0< \ep \leq 1.$  Suppose that the exact solution  $u_{\ep}$ is approximated by the sequence of numerical solutions $\{ U^N_{\ep}\}_{N=1}^{\infty}$, defined on meshes $ \Omega^N$, where $N$ is the discretization parameter. Then, the numerical solutions $U_{\ep}^N$ are said to converge $\ep$-uniformly to the exact solution $u_{\ep}$, if there exists a positive integer $N_0$, and positive numbers $C$ and $p$, all independent of $N$ and $\ep$, such that, for all $N \geq N_0$,
	\bea
	\sup_{0<\ep \leq 1} |U_{\ep}^N-u_{\ep}|_{\Omega^N} \leq C N^{-p},
	\eea
	where $|.|_{\Omega^N}$ is the maximum norm on $\Omega ^N$.\\
\end{defn}
The following theorem provides a proof of parameter uniform first order convergence on a broad class of non-uniform meshes of the discrete solutions to the continuous solution in the maximum norm.
\begin{thm}
	Let the meshes $\Omega^N$ satisfy the condition $h_j<CN^{-1}$ for all $j.$ Then the numerical solutions $U_{\ep}^{N}$ of $(P_{\ep}^N)$ and the exact solution $u_{\ep}$ of $(P_{\ep})$  satisfy the following $\ep$-uniform error estimate, for all $N \geq 1$,
	\bea
	\sup_{0 < \ep \leq 1} |U_{\ep}^N - u_{\ep}|_{\Omega^N} \leq CN^{-1},
	\eea
	where $C$ is a constant independent of $\ep$ and $N$.
\end{thm}
\begin{proof}
	First, the error is decomposed into smooth and singular components 
	\bea
	U_{\ep}^N -u_{\ep}=(V_{\ep}^N -v_{\ep})+(W_{\ep}^N -w_{\ep}).
	\eea
	which we bound separately. We write $\Omega_j=(t_{j-1},t_j).$
	For the smooth component we have
	\bea
	L^N_{\sigma} (V_j -v_j)=f_j-L^N_{\sigma}v_j=(L_{\ep}-L^N_{\sigma})v_j=\ep (\frac{d}{dt}-\sigma_j D^-)v_j=(\frac{d}{dt}-D^-)v_j+\ep(1-\sigma_j)D^- v_j.
	\eea
	Using Lemma \ref{sharp} it follows that, for all $j$,
	\bea
	|L^N_{\sigma} (V_j-v_j&|\leq&|\ep (\frac{d}{dt}-D^-)v_j|+|\ep(1-\sigma_j)D^- v_j|\\
	&\leq & C\ep(h_j|v|_{2,\bar{\Omega_j}}+\rho_j|v|_{1,\bar{\Omega_j}})\\
	&\leq &  Ch_j\\
	&\leq &  CN^{-1},
	\eea
	which is the required bound on the smooth component of the error.\\
	We now estimate of the singular component of the error.
	We have 
	\[
	L^N_\sigma(w_j-W_j)=L^N_\sigma w_j=(\ep\sigma_j D^{-}w_j+a_jw_j)
	=\frac{a_j}{1-e^{-a_j\rho_j}}[w_j-e^{-a_j\rho_j}w_{j-1}].
	\]
	\begin{lem}\label{tech} 
		Let \[A_j=\min_{\Omega_j}a(t),\;\; \alpha_j =\max_ {\Omega_j}a(t). \] Then,
		\[w_{j-1}e^{-\rho_j A_j}\leq w_j \leq w_{j-1}e^{-\rho_j\alpha_j}.\]
		\begin{proof}
			First we prove the inequality on the left hand side. For $t \in \bar{\Omega_j}$, we introduce the function \[\psi(t)=w(t)-w_{j-1}e^{-A_j (t-t_{j-1})/\ep}.\]
			Since $\psi(0)=0$ and $L_\ep \psi(t) \geq 0,$ the maximum principle for $L_\ep$ gives $\psi(t) \geq 0.$ Putting $t=t_j$ gives the desired result. 
			A similar proof gives the inequality on the right hand side.
		\end{proof}
	\end{lem}
	Using the inequalities in the above lemma,  it follows that
	\bea
	L^N_\sigma(w_j-W_j)\leq\frac{a_j w_{j-1}}{1-e^{-\rho_j a_j}}(e^{-\rho_j \alpha_j}-e^{-\rho_j a_j})\leq \frac{a_j w_{j-1}}{1-e^{-\rho_j a_j}}(e^{-\rho_j \alpha_j}-e^{-\rho_j A_j}),\\
	L^N_\sigma(w_j-W_j)\geq\frac{a_j w_{j-1}}{1-e^{-\rho_j a_j}}(e^{-\rho_j A_j}-e^{-\rho_j a_j})\geq -\frac{a_j w_{j-1}}{1-e^{-\rho_j a_j}}(e^{-\rho_j \alpha_j}-e^{-\rho_j A_j})
	\eea
	and so 
	\[|L^N_\sigma (w_j-W_j)| \leq \frac{a_j w_{j-1}}{1-e^{-\rho_j a_j}}(e^{-\rho_j \alpha_j}-e^{-\rho_j A_j}).\]
	Since $A_j-\alpha _j \leq Ch$ and, for any $p>0,\;\; pe^p \leq Ce^{-p/2},$ we have
	\bea
	|L^N_\sigma (w_j-W_j)| &\leq& \frac{C}{1-e^{-\rho_j \alpha_j}}|\rho_j \alpha_j-\rho_j A_j| e^{-\min(\rho_j \alpha_j -\rho_j A_j)}\\
	&\leq& \frac{C}{1-e^{-\rho_j \alpha_j}}\frac{(A_j-\alpha_j)}{\alpha_j}(\rho_je^{-\rho_j \alpha_j})\\
	&\leq& Ch_j\frac{e^{-\rho_j \alpha_j /2}}{1-e^{-\rho_j \alpha_j}}\\
	&\leq& Ch_j\\
	&\leq& CN^{-1}.
	\eea
	From the stability of $L^N_\sigma$ it follows that 
	\[|W_j-w_j|_{\bar\Omega^N} \leq CN^{-1},\]
	which is the required bound on the singular component of the error.
	The bound on the error is obtained by combining the above bounds on the smooth and singular components, which completes the proof of the theorem.
\end{proof}
\section*{Acknowledgment}The author is grateful to Eugene O'Riordan for his essential collaboration in the proof of the theorem in this paper.
	
\end{document}